\documentclass{article}

\usepackage[utf8]{inputenc}

\usepackage{amsmath}
\usepackage{amsthm}
\usepackage{graphicx}
\usepackage{algorithm}
\usepackage{algorithmic}
\usepackage{amssymb}
\usepackage{geometry}
\usepackage{mathtools}
\usepackage{color}

\newtheorem{theorem}{Theorem}
\newtheorem{definition}{Definition}
\newtheorem{remark}{Remark}
\newtheorem{proposition}{Proposition}
\newtheorem{example}{Example}
\newtheorem{lemma}{Lemma}
\newtheorem{corollary}{Corollary}

\title{Integral inequalities for s-convexity via generalized fractional integrals on fractal sets}
\author{Ohud Almutairi$^{1}$ and Adem Kiliçman$^{2}$  \\Department of Mathematics, University of Hafr Al-Batin, Hafr Al-Batin 31991,Sudia Arabia.\\ Universiti Putra Malaysia, Serdang 43400 UPM, Selangor, Malaysia\\ $^{1}$AhoudbAlmutairi@gmail.com	and $^{2}$akilic@upm.edu.my}
\date{\today}

\begin{document}
	\maketitle
\noindent \textbf{Abstract}.
In this study, we establish a new integral inequalities of Hermite-Hadamard type for $s$-convexity 
 via Katugampola fractional integral. This generalizes the Hadamard fractional integrals and Riemann–Liouville into a single form. We show that the new integral inequalities of Hermite–Hadamard type can be obtained via the Riemann–Liouville fractional integral. Finally, we give some applications to special means.

\section{Introduction}
Fractional calculus, whose applications can be found in many fields of studies including economics, life and physical sciences, as well as engineering, can be considered as one of the modern branches of mathematics \cite{He,Tarasov2019,Dragomir2003} and \cite{Tarasov2018}. Numerus problems of interests from these fields can be analyzed through the fractional integrals, which can also be regarded as an interesting sub-discipline of fractional calculus.  Some of the applications of integral calculus can be seen in the following papers \cite{Dragomir2018,Kirmaci2004,Ozdemir2003,Agarwal}, through which problems in physics, chemistry and population dynamics were studied.
The fractional integrals were extended to include the Hermite-Hadamard type inequalities, which are classically given as follows.

Consider a convex function,  $\psi:E\subseteq\mathbb{R}\to\mathbb{R}$ , $u,v\in E$ with $u\leq v$, if and only if,
\begin{flalign}
\psi\bigg(\frac{u+v}{2}\bigg)\leq\frac{1}{v-u}\int_{u}^{v}\psi(x)dx\leq\frac{\psi(u)+\psi(v)}{2}\label{hh}.
\end{flalign}

Following this, many important generalizations of Hermite-Hadamard inequality were studied \cite{Dragomir1999,DragomirS2018,ozcan2019,almutairi,Dragomir1998,almutairi2019}, some of which were formulated via generalized $s$-convexity, which is defined as follows.

\begin{definition}
 The function $\psi :[u, v] \subset \mathbb{R}_{+} \rightarrow \mathbb{R}^{\alpha}$ is said to be generalized $s$-convex in the second sense if
	\begin{equation*}
	\psi\left(t u+\left(1-t\right) v\right) \leq\left(t\right)^{\alpha s} \psi\left(u\right)+\left(1-t\right)^{\alpha s} \psi\left(v\right).
	\end{equation*}
	This class of function is denoted by $GK_s^{2}$ \rm{(see \cite{Mo})}.
\end{definition}

In \cite{mehreen} Mehran and Anwar studied the Hermite-Hadamard-type inequalities for s-convexity involving generalized fractional integrals. The following definitions of the generalized fractional integrals were given in \cite{Katugampola}.
	\begin{definition}\label{def}
		Suppose $[u, v] \subset \mathbb{R}$ is a finite interval. For order $\alpha >0$, the tow sides of
		Katugampola fractional integrals for $\psi \in X_{c}^{p}(u, v)$ are defined by
		\begin{equation*}
		^\rho I_{u^+}^{\alpha} \psi(x)=\frac{\rho^{1-\alpha}}{\Gamma(\alpha)} \int_{u}^{x}\left(x^{\rho}-t^{\rho}\right)^{\alpha-1} t^{\rho-1} \psi(t) d t,
		\end{equation*}
		and
		\begin{equation*}
		^\rho I_{v^-}^{\alpha} \psi(x)=\frac{\rho^{1-\alpha}}{\Gamma(\alpha)} \int_{x}^{v}\left(t^{\rho}-x^{\rho}\right)^{\alpha-1} t^{\rho-1} \psi(t) d t.
		\end{equation*}
	\end{definition}

\vspace{3 mm}

When improving the results in \cite{mehreen}, Mehran and Anwar used Definition \ref{def} together with the following lemma.

\begin{lemma}\label{lem}
	Suppose that $\psi: \left[u^{\rho}, v^{\rho}\right] \subset \mathbb{R}_{+} \rightarrow \mathbb{R}$ is a differentiable mapping on $\left(u^{\rho}, v^{\rho}\right)$, where $0 \leq u<v$ for $\alpha>0$ and $\rho>0$. If the fractional integrals exist, we get
	\begin{flalign*}
	\frac{\psi\left(u^{\rho}\right)+\psi\left(v^{\rho}\right)}{2}-\frac{\rho^{\alpha} \Gamma(\alpha+1)}{2\left(v^{\rho}-u^{\rho}\right)^{\alpha}}\left[^{\rho} I_{u+}^{\alpha} \psi\left(v^{\rho}\right)+^{\rho} I_{v-}^{\alpha} \psi\left(u^{\rho}\right)\right]\nonumber\\=\frac{\rho\left(v^{\rho}-u^{\rho}\right)}{2} \int_{0}^{1}\left[\left(1-t^{\rho}\right)^{\alpha}-\left(t^{\rho}\right)^{\alpha}\right] t^{\rho-1} \psi^{\prime}\left(t^{\rho} u^{\rho}+\left(1-t^{\rho}\right) v^{\rho}\right) d t.
	\end{flalign*}
\end{lemma}

\noindent
This paper is aimed at establishing some new integral inequalities for generalized $s$-convexity  via Katugampola fractional integrals on fractal sets linked with (\ref{hh}). Some inequalities are presented here for the class of mappinges which their derivatives in absolute value at certin powers are generalized $s$-convexity. Also, we obtained some new inequalities linked with convexity and generalized $s$-convexity via classical integrals as well as  Riemann-Liouville fractional integrals in form of a Corollary. As an application, the inequalities for special means are derived.

\section{Main results}

Hermite-Hadamard inequalites for $s$-convexity via a generalized fractional integral can be wriiten with the aid of the following theorem.

\begin{theorem}\label{th1}
		 Suppose that $\psi :\left[u^{\rho}, v^{\rho}\right] \subset \mathbb{R}_{+} \rightarrow \mathbb{R}^{\alpha}$  is a positive function with $0 \leq u<v$  and $\psi \in X_{c}^{p}\left(u^{\rho}, v^{\rho}\right)$ for $\alpha>0$  and $\rho>0$. If $\psi$ is a generalized $s$-convex function on $[u^{\rho}, v^{\rho}]$, then we obtain
		 \begin{equation}
		 \begin{aligned} 2^{\alpha(s-1)} \psi\left(\frac{u^{\rho}+v^{\rho}}{2}\right) & \leq \frac{\rho^{\alpha} \Gamma(\alpha+1)}{2\left(v^{\rho}-u^{\rho}\right)^{\alpha}}\left[^{\rho} I_{u^+}^{\alpha} \psi\left(v^{\rho}\right)+^{\rho} I_{v^-}^{\alpha} \psi\left(u^{\rho}\right)\right] \\ & \leq\left[\frac{1}{\rho(1+s)}+\alpha \beta(\alpha,\alpha s+1)\right] \frac{\psi\left(u^{\rho}\right)+\psi\left(v^{\rho}\right)}{2}.\end{aligned}\label{minth}
		 \end{equation}
\end{theorem}

\begin{proof}
Since $\psi$ is a generalized $s$-convex function on $[u^{\rho}, v^{\rho}]$, for $t\in[0,1]$, we get
\begin{equation*}
\psi\left(t^{\rho} u^{\rho}+\left(1-t^{\rho}\right) v^{\rho}\right) \leq\left(t^{\rho}\right)^{\alpha s} \psi\left(u^{\rho}\right)+\left(1-t^{\rho}\right)^{\alpha s} \psi\left(v^{\rho}\right),
\end{equation*}
and
\begin{equation*}
\psi\left(t^{\rho} v^{\rho}+\left(1-t^{\rho}\right) u^{\rho}\right) \leq\left(t^{\rho}\right)^{\alpha s} \psi\left(v^{\rho}\right)+\left(1-t^{\rho}\right)^{\alpha s} \psi\left(u^{\rho}\right).
\end{equation*}

Combining the above inequalities, we have
\begin{equation}
\psi\left(t^{\rho} u^{\rho}+\left(1-t^{\rho}\right) v^{\rho}\right)+\psi\left(t^{\rho} v^{\rho}+\left(1-t^{\rho}\right) u^{\rho}\right) \leq\left(\left(t^{\rho}\right)^{\alpha s}+\left(1-t^{\rho}\right)^{\alpha s}\right)\left[\psi\left(u^{\rho}\right)+\psi\left(v^{\rho}\right)\right]\label{mineq}.
\end{equation}
\noindent
Multiplying both sides of inequality (\ref{mineq})  by $t^{\alpha \rho-1}$, for $\alpha>0$  and integrating it  over $[0,1]$ with respect to $t$, we obtain
\begin{equation}
\frac{\rho^{\alpha-1} \Gamma(\alpha)}{\left(v^{\rho}-u^{\rho}\right)^{\alpha}}\left[^{\rho} I_{u^{+}}^{\alpha} \psi\left(v^{\rho}\right)+^{\rho} I_{v^-}^{\alpha} \psi\left(u^{\rho}\right)\right] \leq \int_{0}^{1} t^{\alpha \rho-1}\left(\left(t^{\rho}\right)^{\alpha s}+\left(1-t^{\rho}\right)^{\alpha s}\right)\left[\psi\left(u^{\rho}\right)+\psi\left(v^{\rho}\right)\right] dt\label{ineqmin}.
\end{equation}
Since
\begin{equation*}
\int_{0}^{1}t^{\alpha s \rho+\alpha \rho-1}dt=\frac{1}{\alpha \rho(s+1)},
\end{equation*}
and applying change of variable $t^{\rho}=a$, we get
\begin{equation*}
\int_{0}^{1} t^{\alpha \rho-1}\left(1-t^{\rho}\right)^{\alpha s} d t=\frac{\beta(\alpha, \alpha s+1)}{\rho}.
\end{equation*}
Thus, inequality (\ref{ineqmin}) becomes
\begin{equation*}
\frac{\rho^{\alpha} \Gamma(\alpha+1)}{2\left(v^{\rho}-u^{\rho}\right)^{\alpha}}\left[^{\rho} I_{u^+}^{\alpha} \psi\left(v^{\rho}\right)+^{\rho} I_{v^-}^{\alpha} \psi\left(u^{\rho}\right)\right] \leq\left[\frac{1}{\rho(1+s)}+\alpha \beta(\alpha,\alpha s+1)\right] \frac{\psi\left(u^{\rho}\right)+\psi\left(v^{\rho}\right)}{2}.
\end{equation*}

When proving the first part of inequality (\ref{minth}), we observe that $\psi$ is a generalized $s$-convex function on $\left[u^{\rho}, v^{\rho}\right]$, through wich the following inequality is obtained
	\begin{equation}
	\psi\left(\frac{x^{\rho}+y^{\rho}}{2}\right) \leq \frac{\psi\left(x^{\rho}\right)+\psi\left(y^{\rho}\right)}{2^{\alpha s}}\label{inef},
	\end{equation}
	for $x,y\in [a,b], \alpha\geq 0$.\\

Consider $
x^{\rho}=t^{\rho} u^{\rho}+\left(1-t^{\rho}\right) v^{\rho}
$ and $y^{\rho}=
t^{\rho} v^{\rho}+\left(1-t^{\rho}\right) u^{\rho}$, where $t\in [0,1]$.\\
\noindent
Applying inequality (\ref{inef}), we have
\begin{equation}
2^{\alpha s} \psi\left(\frac{u^{\rho}+v^{\rho}}{2}\right) \leq \psi\left(t^{\rho} u^{\rho}+\left(1-t^{\rho}\right) v^{\rho}\right)+\psi\left(t^{\rho} v^{\rho}+\left(1-t^{\rho}\right) u^{\rho}\right).\label{ineqab}
\end{equation}
 Multiplying both sides of the inequality (\ref{ineqab}) by $t^{\alpha \rho-1}$, for $\alpha>0$ and then integrating  over $[0,1]$ with respect to $t$ give the following
 \begin{equation}
 \begin{aligned} \frac{2^{s}}{\alpha \rho} \psi\left(\frac{u^{\rho}+v^{\rho}}{2}\right) \leq & \int_{0}^{1} t^{\alpha \rho-1} \psi\left(t^{\rho} u^{\rho}+\left(1-t^{\rho}\right) v^{\rho}\right) d t+\int_{0}^{1} t^{\alpha \rho-1} \psi\left(t^{\rho} v^{\rho}+\left(1-t^{\rho}\right) u^{\rho}\right) d t \\=& \int_{v}^{u}\left(\frac{v^{\rho}-x^{\rho}}{v^{\rho}-u^{\rho}}\right)^{\alpha-1} \psi\left(x^{\rho}\right) \frac{x^{\rho-1}}{u^{\rho}-v^{\rho}} d x \\ &+\int_{u}^{v}\left(\frac{y^{\rho}-u^{\rho}}{v^{\rho}-u^{\rho}}\right)^{\alpha-1} \psi\left(y^{\rho}\right) \frac{y^{\rho-1}}{v^{\rho}-u^{\rho}} d y \\=& \frac{\rho^{\alpha-1} \Gamma(\alpha)}{\left(v^{\rho}-u^{\rho}\right)^{\alpha}}\left[I_{u^+}^{\alpha} \psi\left(v^{\rho}\right)+^{\rho} I_{v^-}^{\alpha} \psi\left(u^{\rho}\right)\right]. \end{aligned}
 \end{equation}
 Then it follows that
 \begin{equation*}
2^{\alpha(s-1)} \psi\left(\frac{u^{\rho}+v^{\rho}}{2}\right) \leq \frac{\rho^{\alpha} \Gamma(\alpha+1)}{2\left(v^{\rho}-u^{\rho}\right)^{\alpha}}\left[^{\rho} I_{u^+}^{\alpha} \psi\left(v^{\rho}\right)+^{\rho} I_{v^-}^{\alpha} \psi\left(u^{\rho}\right)\right],
 \end{equation*}
 where $\beta(u, v)$ is the Beta function.
\end{proof}

\begin{remark}
	\rm{When substituting $\rho =1$ and $\alpha=1$ in Theorem \ref{th1}, we reported the results of this theorem to that of Theorem 2.1 (see Dragomir and Fitzpatrick \cite{Dragomir1999})}.
	\end{remark}

\begin{example}
	Consider a function $\psi :\left[u^{\rho}, v^{\rho}\right] \subset \mathbb{R}_{+}\rightarrow \mathbb{R}^{\alpha}$, such that $\psi(x)=x^{s\alpha}$ belongs to $GK_s^2$, $s\in (0,1]$ with $\psi \in X_{c}^{p}\left(u^{\rho}, v^{\rho}\right)$, where $\alpha>0$ and $\rho>0$. Suppose $u=0$ and $v=1$. For $\alpha =2, s=\frac{1}{2}$ and $\rho=1$, the first, second and third parts of \rm{inequalities (\ref{minth})} give $0.25, 0.333$ and $0.5$, respectively. Thus, the inequalities \ref{minth} hold. Similarlary, when $\alpha =1, s=\frac{1}{2}$ and $\rho=2$, we get $0.35355,0.5$ and $0.8$, respectively, which satisfy Theorem \ref{th1}.

\end{example}

In the next theorem, the new upper bound for the right-hand side of (\ref{hh}) for generalized $s$-convexity is proposed. For this recall that the generalized beta function is defined as

\begin{flalign*}
\beta_{\rho}(u,v)=\int_{0}^{1}\rho (1-x^{\rho})^{b-1} (x^{\rho})^{a-1} x^{\rho-1} d x.
\end{flalign*}
Note that as $\rho \rightarrow 1$  then $\beta_{\rho}(u, v) \rightarrow \beta(u, v)$.

\begin{theorem}\label{dth}
 Let $\alpha>0$ and $\rho>0$. Let $\psi :\left[u^{\rho}, v^{\rho}\right] \subset \mathbb{R}_{+} \rightarrow \mathbb{R}^{\alpha}$ be a differentiable function on $(u^{\rho}, v^{\rho})$ whereby $\psi^{\prime} \in L_{1}[u, v]$ with  $0 \leq u<v$. If $|\psi^{\prime}|^{q}$ is a grneralized $s$-convex on $\left[u^{\rho}, v^{\rho}\right]$ for  $q \geq 1$, we get
\begin{flalign*} 
 \bigg|\frac{\psi\left(u^{\rho}\right)+\psi\left(v^{\rho}\right)}{2}-\frac{\rho^{\alpha} \Gamma(\alpha+1)}{2\left(v^{\rho}-u^{\rho}\right)^{\alpha}}\left[^{\rho} I_{u^+}^{\alpha} \psi\left(v^{\rho}\right)+^{\rho} I_{v^-}^{\alpha} \psi\left(u^{\rho}\right)\right]\bigg|&\leq \frac{\rho\left(v^{\rho}-u^{\rho}\right)}{2}\bigg(\frac{1}{(\alpha +1)\rho}\bigg)^{\frac{q-1}{q}}\\&\times\bigg[\frac{{\beta}_{\rho}(\alpha s+1, \alpha+1)}{\rho}+\frac{1}{(\alpha (s+1) \rho+1)}\bigg]^{\frac{1}{q}}\\&\times(|\psi^{\prime}(u^{\rho})|^{q}+|\psi^{\prime}(v^{\rho})|^{q})^\frac{1}{q}.
\end{flalign*} 
\end{theorem}
\begin{proof}
	In view of Lemma \ref{lem}, we have
	\begin{flalign}
	\bigg|\frac{\psi\left(u^{\rho}\right)+\psi\left(v^{\rho}\right)}{2}-\frac{\rho^{\alpha} \Gamma(\alpha+1)}{2\left(v^{\rho}-u^{\rho}\right)^{\alpha}}\left[^{\rho} I_{u^+}^{\alpha} \psi\left(v^{\rho}\right)+^{\rho} I_{v^-}^{\alpha} \psi\left(u^{\rho}\right)\right]\bigg|\nonumber&=\bigg|\frac{\rho\left(v^{\rho}-u^{\rho}\right)}{2} \int_{0}^{1}\left[\left(1-t^{\rho}\right)^{\alpha}-\left(t^{\rho}\right)^{\alpha}\right] t^{\rho-1} \\&\times \psi^{\prime}\left(t^{\rho} u^{\rho}+\left(1-t^{\rho}\right) v^{\rho}\right) d t\bigg|\label{lemgs}.
\end{flalign}
First, suppose that $q=1$. Since $|\psi^{\prime}|$ is a grneralized $s$-convex on $\left[a^{\rho},b^{\rho}\right]$, we have
\begin{equation*}
\psi^{\prime}\left(t^{\rho} u^{\rho}+\left(1-t^{\rho}\right) v^{\rho}\right) \leq\left(t^{\rho}\right)^{\alpha s} \psi^{\prime}\left(u^{\rho}\right)+\left(1-t^{\rho}\right)^{\alpha s} \psi^{\prime}\left(v^{\rho}\right).
\end{equation*}
Therefore,
\begin{small}
\begin{flalign}
\bigg| \int_{0}^{1}\left[\left(1-t^{\rho}\right)^{\alpha}-\left(t^{\rho}\right)^{\alpha}\right] t^{\rho-1} \psi^{\prime}\left(t^{\rho} u^{\rho}+\left(1-t^{\rho}\right) v^{\rho}\right) d t\bigg|\nonumber&\leq \int_{0}^{1}\left[\left(1-t^{\rho}\right)^{\alpha}+\left(t^{\rho}\right)^{\alpha}\right] t^{\rho-1} [\left(t^{\rho}\right)^{\alpha s} |\psi^{\prime}(u^{\rho}|)\nonumber\\&+\left(1-t^{\rho}\right)^{\alpha s} |\psi^{\prime}\left(v^{\rho}|\right)] dt
\nonumber\\&=S_1+S_2.\label{s}
\end{flalign}
\end{small}
Calculating $S_1$ and $S_2$, we get
\begin{flalign}
S_1&=|\psi^{\prime}(u^{\rho}|)\bigg[\int_{0}^{1} (1-t^{\rho})^{\alpha}t^{\rho-1}(t^{\rho})^{\alpha s}dt+\int_{0}^{1} (t^{\rho})^{\alpha (s+1)}t^{\rho-1}dt \bigg]\nonumber\\&=|\psi^{\prime}\left(u^{\rho}\right)|\left[\frac{{\beta_{\rho}}(\alpha s+1, \alpha+1)}{\rho}+\frac{1}{\rho(\alpha s+\alpha +1)}\right],\label{ss}
\end{flalign}
and
\begin{flalign}
S_2&=|\psi^{\prime}(v^{\rho}|)\bigg[\int_{0}^{1} (1-t^{\rho})^{\alpha(s+1)}t^{\rho-1}dt+\int_{0}^{1} (t^{\rho})^{\alpha}t^{\rho-1}(1-t^{\rho})^{\alpha s}dt \bigg]\nonumber\\&=|\psi^{\prime}(v^{\rho})|\bigg[\frac{1}{\rho(\alpha s+\alpha +1)}+\frac{{\beta}_{\rho}(\alpha +1,\alpha s+1)}{\rho}\bigg].\label{sss}
\end{flalign}
Thus, if we use inequalities (\ref{ss}) and (\ref{sss}) in (\ref{s}), we obtain
\begin{small}
	\begin{flalign}
	\bigg| \int_{0}^{1}\left[\left(1-t^{\rho}\right)^{\alpha}-\left(t^{\rho}\right)^{\alpha}\right] t^{\rho-1} \psi^{\prime}\left(t^{\rho} u^{\rho}+\left(1-t^{\rho}\right) v^{\rho}\right) d t\bigg|\nonumber&\leq |\psi^{\prime}\left(u^{\rho}\right)|\left[\frac{{\beta_{\rho}}(\alpha s+1, \alpha+1)}{\rho}+\frac{1}{\rho(\alpha s+\alpha +1)}\right]\\&+|\psi^{\prime}(v^{\rho})|\bigg[\frac{1}{\rho(\alpha s+\alpha +1)}+\frac{{\beta}_{\rho}(\alpha +1,\alpha s+1)}{\rho}\bigg]\label{es}.
	\end{flalign}
\end{small}
The inequalities (\ref{lemgs}) and (\ref{es}) complete the proof for this case.

Consider the second case, $q>1$. Using inequality (\ref{lemgs}) and the power mean inequality, we obtain
\begin{small}
	\begin{flalign}
	\bigg| \int_{0}^{1}\left[\left(1-t^{\rho}\right)^{\alpha}-\left(t^{\rho}\right)^{\alpha}\right] t^{\rho-1} \psi^{\prime}\left(t^{\rho} u^{\rho}+\left(1-t^{\rho}\right) v^{\rho}\right) d t\bigg|\nonumber&\leq \bigg(\int_{0}^{1}|(1-t^{\rho})^{\alpha}-(t^{\rho})^{\alpha}| t^{\rho-1}dt\bigg)^{1-\frac{1}{q}} \\&\times\bigg(\int_{0}^{1}|(1-t^{\rho})^{\alpha}-(t^{\rho})^{\alpha}| t^{\rho-1}|\psi^{\prime}(t^{\rho}u^{\rho}+(1-t^{\rho})v^{\rho})|^{q} dt\bigg)^{\frac{1}{q}}\nonumber\\&\leq \int_{0}^{1}\bigg([(1-t^{\rho})^{\alpha}+(t^{\rho})^{\alpha}] t^{\rho-1}dt\bigg)^{1-\frac{1}{q}}\nonumber \\&\times\bigg(\int_{0}^{1}[(1-t^{\rho})^{\alpha}+(t^{\rho})^{\alpha}] t^{\rho-1}[(t^{\rho})^{\alpha s} |\psi^{\prime}(u^{\rho})|^{q}\nonumber\\&+(1-t^{\rho})^{\alpha s} |\psi^{\prime}(v^{\rho})|^{q}] dt\bigg)^{\frac{1}{q}}\nonumber\\&=\bigg(\frac{1}{\rho(\alpha +1)}\bigg)^{\frac{q-1}{q}}\nonumber\\&\times\bigg(\bigg(\frac{{\beta}_{\rho}(\alpha s+1, \alpha+1)}{\rho}+\frac{1}{\rho(\alpha s+\alpha +1)}\bigg)|\psi^{\prime}(u^{\rho})|^{q}\nonumber\\&+\frac{1}{\rho(\alpha s+\alpha +1)}+\frac{{\beta}_{\rho}(\alpha +1,\alpha s+1))}{\rho}|\psi^{\prime}(v^{\rho})|^{q}\bigg)^{\frac{1}{q}}.\label{eth}
	\end{flalign}
\end{small}
The inequalities (\ref{lemgs}) and (\ref{eth}) complete the proof.
\end{proof}
\newpage
\begin{corollary}\label{corol}
Using the similar assumptions given in Theorem \ref{dth}.
\begin{itemize}
	\item[1.]If $\rho=1$, we get
\begin{flalign*} 
\bigg|\frac{\psi(u)+\psi(v)}{2}-\frac{\Gamma(\alpha+1)}{2(v-u)^{\alpha}}[ I_{u^+}^{\alpha} \psi(v)+I_{v^-}^{\alpha} \psi(u)]\bigg|&\leq \frac{(v-u)}{2}\bigg(\frac{1}{\alpha +1}\bigg)^{\frac{q-1}{q}}\\&\times\bigg[{\beta}(\alpha s+1, \alpha+1)+\frac{1}{\alpha (s+1) +1}\bigg]^{\frac{1}{q}}(|\psi^{\prime}(u)|+|\psi^{\prime}(v)|).
\end{flalign*}  
	\item[2.]If $\rho=1$ and $s=1$, then
	\begin{flalign*} 
	\bigg|\frac{\psi(u)+\psi(v)}{2}-\frac{\Gamma(\alpha+1)}{2(v-u)^{\alpha}}[ I_{u^+}^{\alpha} \psi(v)+I_{v^-}^{\alpha} \psi(u)]\bigg|&\leq \frac{(v-u)}{2}\bigg(\frac{1}{1+\alpha }\bigg)^{\frac{q-1}{q}}\\&\times\bigg({\beta}(\alpha +1, \alpha+1)+\frac{1}{1+2\alpha }\bigg)^{\frac{1}{q}}(|\psi^{\prime}(u)|^{q}+|\psi^{\prime}(v)|^{q}).
	\end{flalign*} 
	\item[3.]If $\rho=1$, $s=1$ and $\alpha=1$, we obtain
		\begin{flalign*} 
	\bigg|\frac{\psi(u)+\psi(v)}{2}-\frac{1}{v-u}\int_{u}^{v}\psi(x)dx\bigg|&\leq \frac{(v-u)}{2}\bigg(\frac{1}{2}\bigg)^{\frac{q-1}{q}}\bigg(\frac{\psi^{\prime}(u)|^{q}+\psi^{\prime}(v)|^{q}}{2}\bigg)^{\frac{1}{q}}.
	\end{flalign*} 
\end{itemize}	
\end{corollary}

\begin{theorem}\label{thss}
With the similar assumptions stated in \rm{Theorem \ref{dth}}, we get the following inequality:
 
\begin{flalign} 
 \bigg|\frac{\psi\left(u^{\rho}\right)+\psi\left(v^{\rho}\right)}{2}-\frac{\rho^{\alpha} \Gamma(\alpha+1)}{2\left(v^{\rho}-u^{\rho}\right)^{\alpha}}\left[^{\rho} I_{u^+}^{\alpha} \psi\left(v^{\rho}\right)+^{\rho} I_{v^-}^{\alpha} \psi\left(u^{\rho}\right)\right]\bigg|\nonumber&\leq \frac{(v^{\rho}-u^{\rho})}{2}\bigg(\frac{1}{\rho}\bigg)^{\frac{q-1}{q}} \\&\times\bigg[\frac{\beta_{\rho}(\alpha s +1,\alpha +1)}{\rho}+\frac{1}{(\alpha (s+1) +1)\rho}\bigg]^{\frac{1}{q}}\nonumber\\&\times\bigg(|\psi^{\prime}(u^{\rho})|^{q}+|\psi^{\prime}(v^{\rho})|^{q}\bigg)^{\frac{1}{q}}.\label{cth}
\end{flalign}	
\end{theorem}

\begin{proof}
	Using the fact $|\psi^{\prime}|^{q}$, a grneralized $s$-convex on $\left[u^{\rho},v^{\rho}\right]$ with $q\geq1$, we get
	\begin{equation}
	\psi^{\prime}\left(t^{\rho} u^{\rho}+\left(1-t^{\rho}\right) v^{\rho}\right)\nonumber \leq\left(t^{\rho}\right)^{\alpha s} \psi^{\prime}\left(u^{\rho}\right)+\left(1-t^{\rho}\right)^{\alpha s} \psi^{\prime}\left(v^{\rho}\right).\label{ines}
	\end{equation}

	Applying inequality (\ref{ines}) together with the power mean inequality, we get
\begin{small}
	\begin{flalign*}
	\bigg| \int_{0}^{1}\left[\left(1-t^{\rho}\right)^{\alpha}-\left(t^{\rho}\right)^{\alpha}\right] t^{\rho-1} \psi^{\prime}\left(t^{\rho} u^{\rho}+\left(1-t^{\rho}\right) v^{\rho}\right) d t\bigg|\nonumber&\leq \bigg( \int_{0}^{1}t^{\rho-1}dt\bigg)^{1-\frac{1}{q}} \\&\times\bigg(\int_{0}^{1}|(1-t^{\rho})^{\alpha}-(t^{\rho})^{\alpha}| t^{\rho-1}|\psi^{\prime}(t^{\rho}u^{\rho}+(1-t^{\rho})v^{\rho})|^{q} dt\bigg)^{\frac{1}{q}}\nonumber\\&\leq \bigg(\frac{1}{\rho}\bigg)^{1-\frac{1}{q}}\nonumber \bigg(\int_{0}^{1}[(1-t^{\rho})^{\alpha}+(t^{\rho})^{\alpha}] t^{\rho-1}[(t^{\rho})^{\alpha s} |\psi^{\prime}(u^{\rho})|^{q}\nonumber\\&+(1-t^{\rho})^{\alpha s} |\psi^{\prime}(v^{\rho})|^{q}] dt\bigg)^{\frac{1}{q}}\nonumber\\&\leq \bigg(\frac{1}{\rho}\bigg)^{1-\frac{1}{q}} \bigg(|\psi^{\prime}(u^{\rho})|^{q} \int_{0}^{1}[(1-t^{\rho})^{\alpha}(t^{\rho})^{\alpha s}t^{\rho-1}+(t^{\rho})^{\alpha}(t^{\rho})^{\alpha s}t^{\rho-1}]dt\nonumber\\&\nonumber+|\psi^{\prime}(v^{\rho})|^{q}\int_{0}^{1}[(1-t^{\rho})^{\alpha}(1-t^{\rho})^{\alpha s}t^{\rho-1}+(t^{\rho})^{\alpha}(1-t^{\rho})^{\alpha s}t^{\rho-1}]dt\bigg)^{\frac{1}{q}}\nonumber\\&=\bigg(\frac{1}{\rho}\bigg)^{1-\frac{1}{q}} \nonumber\\&\times\bigg(|\psi^{\prime}(u^{\rho})|^{q}\bigg[\frac{\beta_{\rho}(\alpha s +1,\alpha +1)}{\rho}+\frac{1}{\rho(\alpha s +\alpha +1)}\bigg] \nonumber\\&\nonumber+|\psi^{\prime}(v^{\rho})|^{q}\bigg[\frac{\beta_{\rho}(\alpha +1,\alpha s +1)}{\rho}+\frac{1}{\rho(\alpha s +\alpha +1)}\bigg]\bigg)^{\frac{1}{q}}.
	\end{flalign*} 
\end{small}
\end{proof}	
\begin{remark}
If we choose $\rho=1$ in Theorem \ref{thss}, then we get the following
	\begin{flalign*} 
	\bigg|\frac{\psi(u)+\psi(v)}{2}-\frac{\Gamma(\alpha+1)}{2(v-u)^{\alpha}}[ I_{u+}^{\alpha} \psi(v)+I_{v-}^{\alpha} \psi(u)]\bigg|&\leq \frac{(v-u)}{2}\\&\times\bigg[{\beta}(\alpha s+1, \alpha+1)+\frac{1}{\alpha (s+1) +1}\bigg]^{\frac{1}{q}}(|\psi^{\prime}(u)|+|\psi^{\prime}(v)|).
	\end{flalign*}
\end{remark}
	\begin{remark}
When choosing $\rho=1$ and $s=\frac{1}{2}$ in Theorem \ref{thss}, we obtain
\begin{flalign*} 
\bigg|\frac{\psi(u)+\psi(v)}{2}-\frac{\Gamma(\alpha+1)}{2(v-u)^{\alpha}}[ I_{u+}^{\alpha} \psi(v)+I_{v-}^{\alpha} \psi(u)]\bigg|&\leq \frac{(v-u)}{2}\bigg({\beta}\bigg(\frac{\alpha}{2} +1, \alpha+1\bigg)+\frac{1}{\frac{3}{2}\alpha +1}\bigg)^{\frac{1}{q}}\\&\times(|\psi^{\prime}(u)|^{q}+|\psi^{\prime}(v)|^{q}).
\end{flalign*}
\end{remark}

\begin{corollary}\label{coq}
Choosing $\rho=1$, $s=1$ and $\alpha=1$in \rm{Theorem \ref{thss}}, we obtain
		\begin{flalign*} 
\bigg|\frac{\psi(u)+\psi(v)}{2}-\frac{1}{v-u}\int_{u}^{v}\psi(x)dx\bigg|\leq \frac{(v-u)}{2}\bigg(\frac{\psi^{\prime}(u)|^{q}+\psi^{\prime}(v)|^{q}}{2}\bigg)^{\frac{1}{q}}.
\end{flalign*} 
\end{corollary}

\newpage
The other type is given by the next theorem.
\begin{theorem}\label{enth}
 Let $\alpha>0$ and $\rho>0$. Let $\psi :\left[u^{\rho}, v^{\rho}\right] \subset \mathbb{R}_{+} \rightarrow \mathbb{R}^{\alpha}$ be a differentiable function on $(u^{\rho}, v^{\rho})$ where $\psi^{\prime} \in L_{1}[u,v]$ with  $0 \leq u<v$. For $q > 1$,  if $|\psi^{\prime}|^{q}$ is a grneralized $s$-convex on  $\left[u^{\rho}, v^{\rho}\right]$, we get
	\begin{flalign*}
 \bigg|\frac{\psi\left(u^{\rho}\right)+\psi\left(v^{\rho}\right)}{2}-\frac{\rho^{\alpha} \Gamma(\alpha+1)}{2\left(v^{\rho}+u^{\rho}\right)^{\alpha}}\left[^{\rho} I_{u+}^{\alpha} \psi\left(v^{\rho}\right)+^{\rho} I_{v-}^{\alpha} \psi\left(u^{\rho}\right)\right]\bigg|\nonumber&\leq \frac{\rho(v^{\rho}-u^{\rho})}{2}\bigg(\frac{1}{p(\rho-1)+1}\bigg)^{\frac{1}{p}}\\&\times\bigg[\frac{\beta_{\rho}(\alpha s +1,\alpha +1)}{\rho}+\frac{1}{\rho(\alpha s +\alpha +1)}\bigg]^{\frac{1}{q}}\\&\times\bigg(|\psi^{\prime}(u^{\rho})|^{q}+|\psi^{\prime}(v^{\rho})|^{q}\bigg)^{\frac{1}{q}},
\end{flalign*}	
	with $\frac{1}{p}+\frac{1}{q}=1$.
\end{theorem}	
\begin{proof}
Using the Hölder's inequality, we obtain
	\begin{flalign}
\bigg| \int_{0}^{1}\left[\left(1-t^{\rho}\right)^{\alpha}-\left(t^{\rho}\right)^{\alpha}\right] t^{\rho-1} \psi^{\prime}\left(t^{\rho} u^{\rho}+\left(1-t^{\rho}\right) v^{\rho}\right) d t\bigg|\nonumber&\leq \bigg( \int_{0}^{1}(t^{\rho-1})^{p}dt\bigg)^{\frac{1}{p}} \\&\times\bigg(\int_{0}^{1}[(1-t^{\rho})^{\alpha}+(t^{\rho})^{\alpha}]t^{\rho-1} |\psi^{\prime}(t^{\rho}u^{\rho}+(1-t^{\rho})v^{\rho})|^{q} dt\bigg)^{\frac{1}{q}}\nonumber.
\end{flalign} 
The fact $|\psi^{\prime}|$ is a generalized $s$-convexity, and it can be used to obtain the following,
\begin{flalign}
\bigg| \int_{0}^{1}\left[\left(1-t^{\rho}\right)^{\alpha}-\left(t^{\rho}\right)^{\alpha}\right] t^{\rho-1} \psi^{\prime}\left(t^{\rho} u^{\rho}+\left(1-t^{\rho}\right) v^{\rho}\right) d t\bigg|\nonumber&\leq \bigg(\frac{1}{1+(\rho-1)p}\bigg)^{\frac{1}{p}}\nonumber \\&\times\bigg(\int_{0}^{1}[(1-t^{\rho})^{\alpha}+(t^{\rho})^{\alpha}]t^{\rho-1} [(t^{\rho})^{\alpha s} |\psi^{\prime}(u^{\rho})|^{q}\nonumber\\&+(1-t^{\rho})^{\alpha s} |\psi^{\prime}(v^{\rho})|^{q}] dt\bigg)^{\frac{1}{q}}\nonumber\\&\leq \bigg(\frac{1}{1+(\rho-1)p}\bigg)^{\frac{1}{p}} \nonumber\\&\times\bigg(|\psi^{\prime}(u^{\rho})|^{q} \int_{0}^{1}[t^{\rho-1}(1-t^{\rho})^{\alpha}(t^{\rho})^{\alpha s}+t^{\rho-1}(t^{\rho})^{\alpha}(t^{\rho})^{\alpha s}]dt\nonumber\\&\nonumber+|\psi^{\prime}(v^{\rho})|^{q}\int_{0}^{1}[t^{\rho-1}(1-t^{\rho})^{\alpha}(1-t^{\rho})^{\alpha s}+t^{\rho-1}(t^{\rho})^{\alpha}(1-t^{\rho})^{\alpha s}]dt\bigg)^{\frac{1}{q}}\\&=\bigg(\frac{1}{1+(\rho-1)p}\bigg)^{\frac{1}{p}} \nonumber\\&\times\bigg(|\psi^{\prime}(u^{\rho})|^{q}\bigg[\frac{\beta_{\rho}(\alpha s +1,\alpha +1)}{\rho}+\frac{1}{(\alpha (s+1)+1)\rho}\bigg] \nonumber\\&\nonumber+|\psi^{\prime}(v^{\rho})|^{q}\bigg[\frac{1}{\rho(\alpha (s+1)+1)}+\frac{\beta_{\rho}(\alpha +1,\alpha s +1)}{\rho}\bigg]\bigg)^{\frac{1}{q}}.
\end{flalign}
\end{proof}

\begin{corollary}
	\rm{From Theoremes \ref{dth}, \ref{thss} and \ref{enth}} for $q>1$, we obtain the following inequality,
		\begin{flalign*}
	\bigg|\frac{\psi\left(u^{\rho}\right)+\psi\left(v^{\rho}\right)}{2}-\frac{\rho^{\alpha} \Gamma(\alpha+1)}{2\left(v^{\rho}+u^{\rho}\right)^{\alpha}}\left[^{\rho} I_{u+}^{\alpha} \psi\left(v^{\rho}\right)+^{\rho} I_{v-}^{\alpha} \psi\left(u^{\rho}\right)\right]\bigg|\leq \min(M_1,M_2,M_3)\frac{(v^{\rho}-u^{\rho})}{2},
	\end{flalign*}	
	where
		\begin{flalign*}
	M_1=\rho\bigg(\frac{1}{\rho(\alpha +1)}\bigg)^{\frac{q-1}{q}}\bigg[\frac{{\beta}_{\rho}(\alpha s+1, \alpha+1)}{\rho}+\frac{1}{((s+1)\alpha +1)\rho}\bigg]^{\frac{1}{q}}(|\psi^{\prime}(u^{\rho})|^{q}+|\psi^{\prime}(v^{\rho})|^{q})^\frac{1}{q},&&
	\end{flalign*}
		\begin{flalign*}
	M_2=\bigg(\frac{1}{\rho}\bigg)^{\frac{q-1}{q}}\bigg[\frac{{\beta}_{\rho}(\alpha s+1, \alpha+1)}{\rho}+\frac{1}{\rho(\alpha (s+1) +1)}\bigg]^{\frac{1}{q}}(|\psi^{\prime}(u^{\rho})|^{q}+|\psi^{\prime}(v^{\rho})|^{q})^\frac{1}{q},&&
	\end{flalign*}
	and
		\begin{flalign*}
M_3=\rho \bigg(\frac{1}{1+(\rho-1)p}\bigg)^{\frac{1}{p}}\bigg[\frac{{\beta}_{\rho}(\alpha s+1, \alpha+1)}{\rho}+\frac{1}{(\alpha (s+1) +1)\rho}\bigg]^{\frac{1}{q}}(|\psi^{\prime}(u^{\rho})|^{q}+|\psi^{\prime}(v^{\rho})|^{q})^\frac{1}{q}.&&
\end{flalign*}	
\end{corollary}
\section{Applications to the special means}
Using the results obtained, we examine some applications to
special means for non-negative numbers $u$ and $v$.

\begin{itemize}
	\item[1.] The arithmetic mean:\\
	$A=A(u,v)=\frac{u+v}{2}$; $u,v\in \mathbb{R},$ with $u,v>0.$
	\item[2.] The logarithmic mean:\\
	$L(u,v)=\frac{v-u}{\log v-\log u}$; $u,v\in \mathbb{R},$ with $u,v>0.$
	\item[3.] The generalized logarithmic mean:\\
	$L_r(u,v)=\bigg[\frac{v^{r+1}-u^{r+1}}{(v-u)(r+1)}\bigg]^{\frac{1}{r}}$; $r\in\mathbb{Z}\setminus\{-1,0\}$ $u,v\in \mathbb{R}$, with $u,v>0.$
	
\end{itemize}
\noindent
Using the results obtained in Section 2, and the above applications of means, we get the following propositiones.
\begin{proposition} Suppose that $n\in \mathbb{Z}$, $|r|\geq 2$ and $u,v\in \mathbb{R}$ where $0<u<v$. Then for $q\geq 1$, we get the following:
	\begin{flalign*}
\bigg|A(u^{r},v^{r})-L_r^{r}(u,v)\bigg|\leq \frac{(v-u)|r|}{2^{\frac{q-1}{q}+1}}A^{\frac{1}{q}}(|u|^{q(r-1)},|v|^{q(r-1)}).
\end{flalign*}
\end{proposition}

\begin{proof}
	This follows from \rm{Corollary \ref{corol}(iii)} applied for $\psi(x)=x^{r}$, we get the required result.
\end{proof}

\begin{proposition} Suppose that $n\in \mathbb{Z}$, $|r|\geq 2$ and $u,v\in \mathbb{R}$, whereby $0<u<v$. Then for $q\geq 1$, we get the following:
	\begin{flalign*}
	\bigg|A(u^{r},v^{r})-L_r^{r}(u,v)\bigg|\leq \frac{(v-u)|r|}{2}A^{\frac{1}{q}}(|u|^{q(r-1)},|v|^{q(r-1)}).
	\end{flalign*}
\end{proposition}

\begin{proof}
	This follows from \rm{Corollary \ref{coq}} applied for $\psi(x)=x^{n}$, we get the required result.
\end{proof}

\begin{proposition} Suppose that  $u,v\in \mathbb{R}$, where $0<u<v$. Then for $q\geq 1$, we get
	\begin{flalign*}
	\bigg|A(u^{-1},v^{-1})-L(u,v)\bigg|\leq \frac{(v-u)|}{2^{\frac{q-1}{q}}+1}A^{\frac{1}{q}}(|u|^{-2q},|v|^{-2q}).
	\end{flalign*}
\end{proposition}

\begin{proof}
	This follows from \rm{Corollary \ref{corol}(iii)} applied for $\psi(x)=\frac{1}{x}$, we get the required result.
\end{proof}

\begin{proposition} Suppose that  $u,v\in \mathbb{R}$, where  $0<u<v$. Then for $q\geq 1$, we get
	\begin{flalign*}
	\bigg|A(u^{-1},v^{-1})-L(u,v)\bigg|\leq \frac{(v-u)|}{2}A^{\frac{1}{q}}(|u|^{-2q},|v|^{-2q}).
	\end{flalign*}
\end{proposition}

\begin{proof}
	This follows from \rm{Corollary \ref{coq} } applied for $\psi(x)=\frac{1}{x}$, we get the required result.
\end{proof}


\begin{thebibliography}{999}
	\bibitem{He}He, J. H. (2018). Fractal calculus and its geometrical explanation. Results in Physics, 10, 272-276.
	\bibitem{Tarasov2019}Tarasov, V. E. (2019). On history of mathematical economics: Application of fractional calculus. Mathematics, 7(6), 509.
		\bibitem{Dragomir2003}Dragomir, S. S., \& Pearce, C. (2003). Selected topics on Hermite-Hadamard inequalities and applications. Mathematics Preprint Archive, 2003(3), 463-817.
		\bibitem{Tarasov2018}Tarasov, V. (2018). Generalized memory: Fractional calculus approach. Fractal and Fractional, 2(4), 23.
		\bibitem{Dragomir2018}Dragomir, S. S. (2018). Inequalities for the Generalized k-g-Fractional Integrals in Terms of Double Integral Means. In Advances in Mathematical Inequalities and Applications (pp. 1-27). Birkhäuser, Singapore.
			\bibitem{Kirmaci2004}K{\i}rmac{\i}, U. S. (2004). Inequalities for differentiable mappings and applications to special means of real numbers and to midpoint formula. Applied Mathematics and Computation, 147(1), 137-146.
				\bibitem{Ozdemir2003}\"Ozdemir, M. E. (2003). A theorem on mappings with bounded derivatives with applications to quadrature rules and means. Applied mathematics and computation, 138(2-3), 425-434.
			
			
			\bibitem{Agarwal} Agarwal, R. P., K{\i}l{\i}\c{c}man, A., Parmar, R. K., \& Rathie, A. K. (2019). Certain generalized fractional calculus formulas and integral transforms involving $(p,q)$-Mathieu-type series. Advances in Difference Equations, 2019(1), 221.

	\bibitem{Dragomir1999}Dragomir, S. S., \& Fitzpatrick, S. (1999). The Hadamard inequalities for s-convex functions in the second sense. Demonstratio Mathematica, 32(4), 687-696.
	\bibitem{DragomirS2018}Dragomir, S. S. (2018). Inequalities of Hermite–Hadamard type for HH-convex functions. Acta et Commentationes Universitatis Tartuensis de Mathematica, 22(2), 179-190.
	\bibitem{ozcan2019} \"Ozcan, S., \& \'{I}\c{s}can. (2019). Some new Hermite–Hadamard type inequalities for s-convex functions and their applications. Journal of Inequalities and Applications, 2019(1), 201.
		\bibitem{almutairi}Almutairi, A., \&  K{\i}l{\i}\c{c}man, A. (2019). New refinements of the Hadamard inequality on coordinated convex function. Journal of Inequalities and Applications, 2019(1), 192.
	\bibitem{Dragomir1998}Dragomir, S. S., and R. P. Agarwal. Two inequalities for differentiable mappings and applications to special means of real numbers and to trapezoidal formula. Applied Mathematics Letters 11, no. 5 (1998): 91-95.
	
	\bibitem{almutairi2019}Almutairi, A., \&  K{\i}l{\i}\c{c}man, A. (2019). New fractional inequalities of midpoint type
	via s-convexity and their application. Journal of Inequalities and Applications, 2019(1), 267.
		\bibitem{Mo}Mo, H., \& Sui, X. (2014). Generalized-convex functions on fractal sets. In Abstract and Applied Analysis (Vol. 2014). Hindawi.
		\bibitem{mehreen}Mehreen, N., \& Anwar, M. (2018). Integral inequalities for some convex functions via generalized fractional integrals. Journal of inequalities and applications, 2018(1), 208.
	
	
	\bibitem{Katugampola}Katugampola, U. N. (2014). A new approach to generalized fractional derivatives. Bull. Math. Anal. Appl, 6(4), 1-15.
\end{thebibliography}
\end{document}